\documentclass[11pt]{amsart}
\usepackage{url,bm,amssymb,amscd,amsmath}
\usepackage{graphics,epsfig,color}

\usepackage[letterpaper,asymmetric,left=1.25in,right=1.25in,top=1.25in,bottom=1.
25in,bindingoffset=0.0in]{geometry}

\usepackage{url,bm,amssymb,amscd,amsmath}
\usepackage{graphics,epsfig,color}

\usepackage[curve,color,graph]{xypic}

\renewcommand{\tilde}{\widetilde}

\renewcommand{\P}{\mathbb P}
\newcommand{\R}{\mathbb R}
\newcommand{\pr}{{\rm pr}}

\newcommand{\cupprod}{\mathbin{\smile}}
\newcommand{\capprod}{\mathbin{\frown}}

\def\pr{{\rm pr}}
\def\id{{\rm id}}
\def\PD{{\rm PD}}

\newtheorem{theorem}{\bf Theorem}[section]
\newtheorem{proposition}[theorem]{\bf Proposition}

\newtheorem*{theorem*}{\bf Theorem}
\newtheorem*{proposition12}{\bf Proposition 1.2}
\newtheorem*{proposition14}{\bf Proposition 1.4}
\newtheorem*{proposition15}{\bf Proposition 1.5}
\newtheorem*{proposition16}{\bf Proposition 1.6}
\newtheorem{lemma}[theorem]{\bf Lemma}

\theoremstyle{remark}
\newtheorem{remark}[theorem]{\bf Remark}
\newtheorem{example}[theorem]{\bf Example}
\newtheorem{question}[theorem]{\bf Question}

\begin{document}

\title{The action on cohomology by compositions of rational maps}

\begin{author}[R. K. W. Roeder]{Roland K. W. Roeder}
 \address{Department of Mathematical Sciences \\ IUPUI \\ LD Building, Room 224Q\\
402 North Blackford Street\\
Indianapolis, Indiana 46202-3267\\
 United States }
\email{rroeder@math.iupui.edu}
\end{author}


\date{\today}

\begin{abstract}We use intuitive results from algebraic topology and intersection theory to clarify the
pullback action on cohomology by compositions of rational maps.  We use these techniques
to prove a simple sufficient criterion for functoriality of a composition of two rational maps  on all degrees of cohomology
and we then reprove the criteria of Diller-Favre, Bedford-Kim, and Dinh-Sibony.
We conclude with a cautionary example.
\end{abstract}

\maketitle

\section{Introduction}

Suppose that $X$ and $Y$ are complex projective algebraic manifolds, both of dimension
$k$, and $f: X \dashrightarrow Y$ is a rational map.  If $I_f$ denotes the
indeterminacy set of $f$, the graph of $f$ is the irreducible variety 
\begin{eqnarray}\label{EQN:GRAPH}
\Gamma_f := \overline{\{(x,y) \in X \times Y \, : \, x \not \in I_f \, \mbox{and} \, y = f(x) \}}.
\end{eqnarray}
One defines the action $f^*: H^*(Y) \rightarrow H^*(X)$ on the singular cohomology of $X$ by considering $f$ as the correspondence $\Gamma_f \subset X \times Y$.
If $\pi_1: X \times Y \rightarrow X$ and $\pi_2: X \times Y \rightarrow Y$ are the canonical projections,
then, for any $\alpha \in H^i(Y)$,
\begin{eqnarray}\label{EQN:PULLBACK}
f^* \alpha := \pi_{1*}([\Gamma_f] \cupprod \pi_2^* \alpha).
\end{eqnarray}
Here, $[\Gamma_f]$ is the fundamental cohomology class of $\Gamma_f$, $\pi_2^*$ is the classical pullback on cohomology as defined for regular maps, and
$\pi_{1*}$ is the pushforward on cohomology, defined by $\pi_{1*} = \PD_X^{-1}
\circ \pi_{1\#} \circ \PD_{X \times Y}$, where $\pi_{1\#}$ denotes the push
forward on homology and $\PD_M: H^*(M) \rightarrow H_{{\rm dim}_\R M-*}(M)$ denotes the
Poincar\'e duality isomorphism on a manifold $M$.  If $f$ is regular (i.e. $I_f =
\emptyset$) then (\ref{EQN:PULLBACK}) coincides with the classical definition
of pullback. 

We will take the coefficients for our cohomology in $\mathbb{C}$, letting
$H^i(X) \equiv H^i(X;\mathbb{C})$.  Since our manifolds are K\"ahler, there is a natural isomorphism
$$\bigoplus_{p+q = i} H^{p,q}(X)~\rightarrow~H^i(X),$$ where the former are the
Dolbeault cohomology groups.  This isomorphism induces a splitting of the
singular cohomology of $X$ into bi-degrees, which one can check is invariant
under the pullback (\ref{EQN:PULLBACK}). 

The most primitive dynamical invariants of any rational selfmap $h: X \dashrightarrow X$ are the {\em dynamical degrees}
\begin{eqnarray}\label{EQN:DEF_DYN_DEG}
\lambda_p(h) := \lim_{n \rightarrow \infty} \left||(h^n)^*: H^{p,p}(X) \rightarrow  H^{p,p}(X) \right||^{1/n},
\end{eqnarray}
which are defined for $1 \leq p \leq k = \dim(X)$.  They were introduced by
Friedland \cite{FRIEDLAND} and by Russakovskii and Shiffman 
\cite{RUSS_SHIFF} and shown to be invariant under birational conjugacy by Dinh
and Sibony \cite{DS_BOUND}.  Note that dynamical degrees were originally defined with the limit in (\ref{EQN:DEF_DYN_DEG})
replaced by a 
${\rm lim sup}$.
However, it was shown in \cite{DS_BOUND} that the limit always exists.

The dynamical degrees of $h$ are tied to the expected ergodic properties of
$h$; see, for example, \cite{GUEDJ_ERGODIC}.  (These expected properties have
been proved when  $\lambda_k(h)$ is maximal \cite{GUEDJ,DNT} or when $\dim(X) =
2$, $\lambda_1(h) > \lambda_2(h)$, and certain minor technical hypotheses are
satisfied \cite{DDG}.) Dynamical degrees are typically hard to compute because
(\ref{EQN:PULLBACK}) does not behave well under composition of maps.  There are
simple examples for which $(h^n)^* \neq (h^*)^n$.  One says that $h$ is {\em
$p$-stable} if $(h^n)^* = (h^*)^n$ on $H^{p,p}(X)$ for every $n \in
\mathbb{Z}^+$.  A nice summary of techniques on how to compute dynamical
degrees appears in \cite{BEDFORD_DYN_DEG}.  Let us note that there are very few
explicit examples \cite{AMERIK,FAVRE_WULCAN,LIN,LIN_WULCAN} in which the
$p$-th dynamical degrees have been computed for $1 < p  < k$.

In order to study the problem of $p$-stability, one typically looks for
criterion on $f: X \dashrightarrow Y$ and $g: Y \dashrightarrow Z$ under which
$(g \circ f)^* = f^* \circ g^*$ (either on all cohomology or for certain degrees).
Such criteria have been given by Fornaess-Sibony \cite{FORNAESS_SIBONY},
Diller-Favre \cite{DILLER_FAVRE}, Bedford-Kim \cite{BEDFORD_KIM,BEDFORD_KIM2}, and
Dinh-Sibony \cite{DINH_SIBONY}.  The proofs of these criteria typically
represent a cohomology class $\alpha \in H^*(Z)$ with a smooth form, pull it
back under $g^*$ as a closed current, and then pull back the resulting current
under $f^*$.  This approach is especially challenging  when $p \geq 2$ since
the pullback of such higher-codimension currents is very delicate.

\vspace{0.1in}
The purpose of this note is to prove these criteria
using intuitive techniques from cohomology and intersection
theory.  This approach is inspired by the techniques used by Amerik in
\cite{AMERIK}.  Our primary motivation is to provide those who are learning these results with
an alternative approach, in the hope that seeing two different proofs makes the results clearer.

Another merit of this approach is that it may be possible to adapt it to
problems about rational maps between projective manifolds defined over fields
$K \neq \mathbb{C}$.    The dynamics of such mappings has gained considerable
interest recently (see, for example, \cite{AMERIK2,JONSSON_WULCAN,SILVERMAN_KAWAGUCHI,TRUONG2,SILVERMAN1} and the references therein) and the analytic techniques involving smooth forms and
positive closed currents from
\cite{DILLER_FAVRE,BEDFORD_KIM,BEDFORD_KIM2,DINH_SIBONY} do not apply in that context.
However, Intersection Theory (our main underlying tool) still applies to 
projective manifolds defined over other fields $K$.

\vspace{0.1in}

Several of the references listed above consider the broader context of
meromorphic maps of compact K\"ahler manifolds.  In order for the techniques
used in this note to be as elementary as possible, 
we will restrict our attention to
rational maps of projective algebraic manifolds.   This allows us to use classical
techniques from intersection theory, such as Fulton's Excess Intersection
Formula, which will be helpful when establishing Lemma \ref{LEM:CUP_PRODUCT}, below.

Let us make the convention that all rational maps are {\em dominant}, meaning
that the image is not contained within a proper subvariety of the
codomain.  To be concise, we will use the term {\em algebraic manifold} to mean complex projective algebraic manifold.
Moreover, since we are primarily motivated by dynamics, all rational
mappings will be between algebraic manifolds of the same dimension.  For any $S \subset X$, we  define
$f(S):= \pi_2(\pi_1^{-1}(S) \cap \Gamma_f)$ and for any $S \subset Y$ we define $f^{-1}(S):= \pi_1(\pi_2^{-1}(S) \cap \Gamma_f)$

For simplicity of exposition, we will ignore the decomposition of cohomology into bidegree wherever possible.

In order to study the composition $g\circ f$ we will need the following diagram:
\begin{eqnarray}\label{MASSIVE_DIAGRAM}
\xymatrix{
& & \Gamma_{g \circ f} \subset X \times Z \ar@/_3pc/[dddll]_{\pr_1}  \ar@/^3pc/[dddrr]^{\pr_2}  \\
& &  X \times Y \times Z \ar[u]^{\rho_2} \ar[dl]_{\rho_1} \ar[dr]^{\rho_3} & & \\
& \Gamma_f \subset  X \times Y \ar[dl]_{\pi_1} \ar[dr]^{\pi_2} & &  \Gamma_g \subset Y \times Z \ar[dl]_{\pi_3} \ar[dr]^{\pi_4} & \\
X \ar @{-->}[rr]^{f} & & Y \ar @{-->}[rr]^{g} & & Z
}
\end{eqnarray}
Central to the entire discussion is the following.
\begin{proposition}\label{PROP:FUNCTOR}
We have
\begin{eqnarray}\label{EQN:COMPOSED_PULLBACK0}
f^* g^* \alpha = \pr_{1*} \left(\rho_{2*}(\rho_1^*[\Gamma_f] \cupprod \rho_3^* [\Gamma_g]) \, \cupprod \pr_2^* \alpha\right).
\end{eqnarray}
In particular, $(g\circ f)^* = f^* \circ g^*$ on all cohomology groups if and only if
\begin{eqnarray*}
[\Gamma_{g\circ f}] = \rho_{2*}(\rho_1^*[\Gamma_f] \cupprod \rho_3^* [\Gamma_g]) \in H^{2k}(X \times Z).
\end{eqnarray*}
\end{proposition}
\noindent
This proposition is probably well-known within algebraic geometry, for example a variant of (\ref{EQN:COMPOSED_PULLBACK0}) is proved for the pull back on the Chow Ring in \cite[Sec. 16.1]{FULTON} and \cite[Prop. 9.7]{VOISIN2}, but it seems to be less well-known in rational dynamics. 

Our first application of Proposition \ref{PROP:FUNCTOR} is to prove:
\begin{proposition}\label{PROP:FINITE_FIBERS} Let $f: X \dashrightarrow Y$ and $g: Y \dashrightarrow Z$ be rational maps.  
Suppose that there exits an algebraic manifold $\tilde{X}$ and holomorphic maps $\pr$ and $\tilde{f}$ making the following 
diagram commute (wherever $f \circ \pr$ is defined)
\begin{eqnarray} \label{GENERAL_RESOLUTION}
\xymatrix{\tilde{X} \ar[d]^\pr  \ar[dr]^{\tilde{f}}  & \\
X \ar @{-->}[r]^f & Y }
\end{eqnarray}
with the property that $\tilde{f}^{-1}(x)$ is a finite set for every $y \in Y$.  Then, $(g\circ f)^* = f^* \circ g^*$ on all cohomology groups.
\end{proposition}

\begin{remark}
In many cases, $\tilde{X}$ will be a blow-up of $X$.  However this is not a hypothesis of Proposition \ref{PROP:FINITE_FIBERS}, which can also be useful in other situations.  Notice also that the condition that $\tilde{f}^{-1}(x)$ is a finite set implies that $\dim(\tilde{X}) = k$.
\end{remark}

After proving Proposition \ref{PROP:FINITE_FIBERS}, we will use Proposition
\ref{PROP:FUNCTOR} to prove the criteria of Diller-Favre, Bedford-Kim, and
Dinh-Sibony stated below.

Historically, the first criterion for functoriality of pullbacks under
compositions  was given by Fornaess and Sibony \cite{FORNAESS_SIBONY} who
proved that if $f: \mathbb{CP}^k \dashrightarrow \mathbb{CP}^k$ and $g:
\mathbb{CP}^k \dashrightarrow \mathbb{CP}^k$ are rational maps then $(g \circ
f)^* = f^* \circ g^*$ on the second cohomology if and only if there is no
hypersurface $H \subset \mathbb{P}^k$ with $f(H \setminus I_f) \subset I_g$.
The proof consists of recognizing that the homogeneous expression obtained when
composing $f$ and $g$ has a common factor of positive degree if and only if
there is a hypersurface $H \subset \mathbb{P}^k$ with $f(H \setminus I_f)
\subset I_g$.  Since this common factor must be removed in order to define $g
\circ f$, the resulting composition has lower degree.   A further study of this
phenomenon and a characterization of the sequences of degrees that may appear
for the iterates of such a map $f$ is given in \cite{BF}.

Since $I_g$ is of codimension at least two, in order that $f(H \setminus I_f) \subset I_g$, $f$ must collapse $H$ to a variety of lower dimension.
The principle that non-functoriality is caused by collapse of a subvariety under $f$ to something of lower dimension that is contained within $I_g$ appears
as a common theme in the following three criteria:

\begin{proposition}\label{PROP:DF} {\rm (Diller-Favre \cite[Prop. 1.13]{DILLER_FAVRE})} Let $X,Y,$ and $Z$ be algebraic manifolds of dimension $2$. Let $f: X \dashrightarrow Y$ and $g: Y \dashrightarrow Z$ be rational maps.  Then
$(g\circ f)^* = f^* \circ g^*$ if and only if there is no curve $C \subset X$
with $f(C \setminus I_f) \subset I_g$.   \end{proposition}

\begin{proposition}\label{PROP:BK}{\rm (Bedford-Kim \cite[Thm. 1.1]{BEDFORD_KIM2})} Let $X,Y,$ and $Z$ be algebraic manifolds of dimension $k$.
Let $f: X \dashrightarrow Y$ and $g: Y \dashrightarrow Z$ be rational maps.   If there is no hypersurface $H$
with $f(H \setminus I_f) \subset I_g$, then $f^* \circ g^* = (g\circ f)^*$ on $H^2(Z)$.
\end{proposition}

We will prove a slightly stronger variant of the criterion of Dinh and Sibony.
Let $\tilde{\Sigma}'_f \subset \Gamma_f$ is the set of points such that
\begin{enumerate}
\item[(i)] $\pi_2$ restricted to $\Gamma_f$ is not locally finite at $x$, and
\item[(ii)] $\pi_2((x,y)) \in I_g$ for every $(x,y) \in \tilde{\Sigma}'_f$.
\end{enumerate}
 Let $\Sigma'_f:=\pi_1\left(\tilde{\Sigma}'_f\right)$. 

\begin{proposition} \label{PROP:DS}
{\rm (Variant of Dinh-Sibony \cite[Prop. 5.3.5]{DINH_SIBONY}) } Let $X,Y,$ and $Z$ be algebraic manifolds of dimension $k$.
Let $f: X \dashrightarrow Y$ and $g: Y \dashrightarrow Z$ be rational maps.  If $\dim \Sigma'_f < k-p$, then $(g\circ f)^* = f^* \circ g^*$
on $H^i(Z)$ for $1 \leq i \leq 2p$.
\end{proposition}

\begin{remark}The distinction between this criterion and the one from \cite[Prop. 5.3.5]{DINH_SIBONY} is that we impose the extra condition (ii) on  $\Sigma_f$,
allowing for higher dimensional varieties to be collapsed by $f$, so long as they don't map into $I_g$.
\end{remark}

\begin{remark}
Proposition \ref{PROP:FINITE_FIBERS}, the sufficiency condition in Proposition \ref{PROP:DF}, and Proposition
\ref{PROP:BK} can all be obtained as corollaries to Proposition \ref{PROP:DS}.  However,
we'll present them separately since they're of independent interest and their direct proofs are simpler.
\end{remark}

In \S \ref{SEC:BACKGROUND} we provide a brief background with needed tools
from cohomology and intersection theory.  In \S \ref{SEC:SINGULAR_GRAPHS} we
discuss some further properties of the graph $\Gamma_f$ and we show that
definition (\ref{EQN:PULLBACK}) of $f^*$ is equivalent with some of the other
standard versions appearing in the literature.  We  prove Propositions
\ref{PROP:FUNCTOR} in \S \ref{SEC:FUNCTOR}.  In \S \ref{SEC:CRITERIA}  we
prove Propositions \ref{PROP:FINITE_FIBERS}-\ref{PROP:DS}.  This paper is
concluded with \S \ref{SEC:CONCLUSION} in which we provide a cautionary
example, presenting a rational map $f: X \dashrightarrow X$ of a three
dimensional manifold $X$ that is not $2$-stable but has the property that
$\left(f|_{X \setminus I_f}\right)^{-1}(x)$ is finite for every $x \in X$.
This example illustrates that to study $p$ stability for $1 < p < k$, one must
consider collapsing behavior lying within the indeterminate set.

\section{Background from cohomology and intersection theory}\label{SEC:BACKGROUND}

Suppose $f: M \rightarrow N$ is a continuous map between compact manifolds of dimensions $m$ and $n$, respectively.
Given $\alpha \in H^i(M)$, we define $f_*: H^i(M) \rightarrow H^{n-m+i}(N)$ by
\begin{eqnarray}
f_* \alpha := {\rm PD}_N^{-1}(f_\# ({\rm PD}_M \alpha)),
\end{eqnarray}
where $f_\#: H_*(M) \rightarrow H_*(N)$ is the push forward on homology. 

We will make extensive use of the following formula.

\begin{lemma}[Push-Pull Formula]\label{LEM:PUSH_PULL}
Suppose $M$ and $N$ are manifolds and $f: M \rightarrow N$ is continuous.  Then, for any $\alpha \in H^i(N)$ and any $\beta \in H^j(M)$ we have
\begin{eqnarray*}
f_*(f^*(\alpha) \cupprod \beta) = \alpha \cupprod f_*(\beta) \in H^{n-m+i+j}(N).
\end{eqnarray*}
\end{lemma}
\noindent
Note that when $f$ is holomorphic, this is sometimes also called the ``projection formula''.

\begin{proof}
This is a simple consequence of the following three facts 
\begin{itemize}
\item[(i)] Push-Pull formula on homology:  
If $f: M \rightarrow N$ is continuous, $\eta \in H^*(N),$ and $\gamma \in H_*(M)$, then
\begin{eqnarray*}
f_\#(f^*(\eta) \capprod \gamma) = \eta \capprod f_\# \gamma,
\end{eqnarray*}
\item[(ii)] $\PD_M(\alpha)$ is defined by $\alpha \capprod \{M\}$, where $\{M\}$ is the fundamental homology class of $M$, and
\item[(iii)] for any $\eta, \phi \in H^*(M)$ and $\gamma \in H_*(M)$, then $(\eta \cupprod \phi) \capprod \gamma = \eta \capprod (\phi \capprod \gamma)$.
\end{itemize}
See \cite[Ch. VI, Thm. 5.1 and Cor. 9.3]{BREDON}.
\end{proof}



We will need a little bit of information about the K\"unneth formuli on cohomology and homology.
Recall that our (co)homology is taken with coefficients in the field $\mathbb{C}$.
Let $$\kappa_i: \bigoplus_{a+b=i} H^a(M) \otimes H^b(N) \rightarrow H^i(M\times N)$$ and \\
$$K_i: \bigoplus_{a+b=i} H_{a}(M) \otimes H_{b}(N) \rightarrow H_{i}(M \times N)$$ be the K\"unneth isomorphisms.
Recall  that $\kappa_i(\gamma \otimes \eta) = \pi_1^* \gamma \cupprod \pi_2^* \eta$.
Suppose $M$ and $N$ are manifolds.

\begin{lemma}\label{LEM:KUNNETH}
The following diagram commutes:
{\small
\begin{align}\label{KUNNETH_DIAGRAM}
\xymatrix{
\bigoplus_{a+b=i} H^a(M) \otimes H^b(N) \ar[d]^{\kappa_i} \ar[rrrr]^{(-1)^{mb} \PD_M \otimes \PD_N} & & & & \bigoplus_{a+b=i} H_{m-a}(M) \otimes H_{n-b}(N) \ar[d]^{K_{m+n-i}} \\
H^{i}(M \times N) \ar[rrrr]^{\PD_{X\times Y}} & & & & H_{m+n-i}(M\times N).
}
\end{align}
}
\end{lemma}

\begin{proof}
According to \cite[Ch. VI, Thm. 5.4]{BREDON}, if $\alpha \in H^*(X), \beta \in H^*(Y), c \in H_*(X),$ and $d \in H_*(Y)$, then
\begin{eqnarray*}
\kappa(\alpha \otimes \beta) \capprod K(c \otimes d) = (-1)^{\deg(\beta) \deg(c)}K((\alpha \capprod c) \otimes (\beta \capprod d)).
\end{eqnarray*}
The result follows, since $\PD_{M \times N}$ is obtained by taking the cap product with $\{M \times N\}$ \\ $ = K(\{M\} \otimes \{N\})$.
  \end{proof}

\begin{lemma}\label{LEM:KUNNETH_PROJECTION}
Let $M$ and $N$ be connected manifolds of dimensions $m$ and $n$, respectively and
let $\pi: M \times N \rightarrow M$ be projection onto the first coordinate.   
Suppose $\nu \in H^i(M \times N)$ satisfies
\begin{eqnarray*}
\kappa^{-1}(\nu) =  \sum_{a=1}^i \sum_{l=1}^{l_j} \gamma_{i-a,l} \otimes \eta_{a,l}
\end{eqnarray*}
with $\gamma_{i-a,l} \in H^{i-a}(M)$, $\eta_{a,l} \in H^{a}(N)$, and with the normalization that
each  $\eta_{n,l} \in H^n(N)$ is the fundamental class $[x]$ of a point $x \in N$.   Then
\begin{eqnarray*}
\pi_* \nu = (-1)^{mn} \sum_{l=1}^{l_n}\gamma_{i-n,l}.
\end{eqnarray*}
\end{lemma}

\begin{proof}
This follows from Lemma \ref{LEM:KUNNETH} and the fact that the push forward $\pr_\#$ on homology satisfies that
\begin{eqnarray*}
\pr_\#\left(K\left(\sum_{a=1}^i \sum_{l=1}^{l_a} g_{i-a,l} \otimes e_{a,l}\right)\right)
= \sum_{l=1}^{l_0} g_{i,l},
\end{eqnarray*}
if each $g_{i-a,l} \in H_{i-a}(M)$, each $e_{a,l} \in H_a(N)$, and each
$e_{0,i} = \{x\}$ is the fundamental homology class of a point.  This follows
easily from the fact that the K\"unneth Isomorphism is natural with respect to
induced maps.
\end{proof}

\begin{remark}
In our applications, $M$ and $N$ will be complex manifolds.  Since they have even real-dimension, the signs will
disappear from Lemmas \ref{LEM:KUNNETH} and \ref{LEM:KUNNETH_PROJECTION}.
\end{remark}

\vspace{0.1in}

Let $X$ be an algebraic manifold of (complex) dimension $k$ and let $V \subset X$ be a
subvariety of dimension $k-i$.  It is well known that $V$ generates a cohomology
class $[V] \in H^{2i}(X)$; see, for example, \cite{GH,VOISIN1}.  If $V'
\subset X$ is another subvariety of dimension $k-j$, we will need information
relating $[V] \cupprod [V']$ to $V \cap V'$.  This is the subject of Intersection
Theory \cite{FULTON_INTRO,FULTON,VOISIN2}.    One says that $V$ and $V'$ are
{\em transverse at generic points} of $V \cap V'$ if there is a dense set of $V
\cap V'$ on which $V$ and $V'$ are both smooth and 
intersect transversally.  The information we need is encapsulated in:
\begin{lemma}\label{LEM:CUP_PRODUCT}
Let $V$ and $V'$ be subvarieties of $X$ of dimensions $k-i$ and $k-j$.  Then,
$[V] \cupprod [V']$ is represented as a linear combination of fundamental cohomology classes of $k-i-j$-dimensional subvarieties of $V \cap V'$.  
More specifically:
\begin{enumerate}
\item[(i)] If $V$ and $V'$ are transverse at generic points of $V \cap V'$, then $[V] \cupprod [V'] = [V \cap V']$.
\item[(ii)] More generally, if each of the components $W_1,\ldots,W_{m_0}$ of $V \cap V'$ has the correct dimension of $k-i-j$, 
then 
\begin{eqnarray*}
[V] \cupprod [V'] = \sum_{m=1}^{m_0} a_m [W_m],
\end{eqnarray*}
where each $a_m \in \mathbb{Z}^+$ is an {\em intersection number} satisfying that $a_m = 1$ if and only if $V$ and $V'$ are transverse at generic points
of $W_m$.
\item[(iii)] Most generally, if some of the components $W_1,\ldots,W_{m_0}$ of $V \cap V'$ are of dimension $> k-i-j$,
then, 
\begin{eqnarray*}
[V] \cupprod [V'] = \sum_{m=1}^{m_0} \sum_{n=1}^{n_m} a_{m,n} [W_{m,n}],
\end{eqnarray*}
where each $W_{m,n} \subset W_m$ is a subvariety of $W_m$ of dimension $k-i-j$ and each $a_{m,n} \in \mathbb{Z}$.  For each $W_m$
of the correct dimension $k-i-j$ the inner sum reduces to be $a_m [W_m]$, where $a_m$ is given as in {\rm (ii)}.
\end{enumerate}
\end{lemma}

Note that in case (iii), the coefficients $a_{m,n}$ can be negative, for example the self-intersection of the exceptional divisor resulting
from a blow-up of $\mathbb{CP}^2$ is represented by a single point on the exceptional divisor with coefficient $-1$.

\vspace{0.1in}
Rather than presenting a proof of Lemma \ref{LEM:CUP_PRODUCT}, we will mention how to obtain it from the
corresponding properties in the Chow Ring $CH^*(X)$, which are proved in \cite{FULTON_INTRO,FULTON,VOISIN2}.  
For each $0
\leq i \leq k$, the chow group $CH^i(X)$ is the collection of finite formal
sums of $k-i$-dimensional irreducible subvarieties taken with integer
coefficients, up to an equivalence relation known as {\em rational
equivalence}.  We won't need the detailed definition of rational equivalence, however let us
denote the rational equivalence class of an irreducible subvariety $V$
by $(V)$.  

One obtains the Chow Ring $CH^*(X) = \bigoplus_{i=0}^k CH^i(X)$ by
defining an intersection product 
\begin{eqnarray*}
\bullet: CH^i(X) \times CH^j(X) \rightarrow CH^{i+j}(X).
\end{eqnarray*}
If $V$ and $V'$ intersect properly, with dimension $k-i-j$, then each
component of the intersection is assigned an intersection multiplicity in a
relatively simple way, see \cite[Sec. 8.2]{FULTON}.  (Note that using the
uniqueness described in \cite[Eg. 11.4.1]{FULTON}, one can show that this
intersection multiplicity is consistent with the more intuitive approach of
\cite[Sec. 12.3]{CHIRKA}.)  This intersection multiplicity is a positive integer that equals  $1$ if and only if $V$ and $V'$ are generically transverse along the component.

If the intersection has a component whose dimension
is larger than $k-i-j$, there are two approaches: \begin{itemize} \item[(i)]
moving one of the subvarieties $V$ to a rationally equivalent one $\tilde{V}$
in such a way that $\tilde{V} \cap V'$ has the correct dimension, via Chow's
moving lemma (see \cite[Lem. 9.22]{VOISIN2} or \cite[Sec. 11.4]{FULTON}), or
\item[(ii)] or  Fulton's {\em excess intersection formula}, which represents the
intersection product as a linear combination of subvarieties lying within $V
\cap V'$ (see \cite[Sec. 9.2]{VOISIN2} or \cite[Sec. 6.3]{FULTON}).
\end{itemize} 
In order to guarantee the property that the cup product is represented by a sum of fundamental classes of subvarieties of
$V \cap V'$, we appeal to the latter.

Lemma \ref{LEM:CUP_PRODUCT} then follows from the fact that there is a ring
homomorphism ${\rm cl}: CH^*(X) \rightarrow H^{2*}(X)$ with the property
that for any irreducible $V \subset X$, ${\rm cl}((V)) = [V]$. See, for
example, \cite[Ch. 19]{FULTON} or \cite[Lem. 9.18 and Prop. 9.20]{VOISIN2}.

\begin{remark}
In many of our applications, we will only need properties (i) and (ii) which are relatively simple.  We will only use
property (iii) to show that the cup product is given by subvarieties of the geometric intersection $V \cap V'$.  We won't use any details of how the coefficients $a_{m,n}$ in Part (iii) of Lemma \ref{LEM:CUP_PRODUCT} are actually computed.
\end{remark}

\begin{lemma}\label{LEM:PUSH_FORWARD_VARIETIES}
Suppose that $f: X \rightarrow Y$ is a proper holomorphic map between algebraic manifolds.  For any irreducible subvariety $V \subset X$
we have 
\begin{itemize}
\item[(i)] if $\dim(f(V)) = \dim(V)$, then $f_*([V]) = {\rm deg}_{\rm top}(f|_{V}) [f(V)]$, 
where ${\rm deg}_{\rm top}(f|_{V})$ is the number of preimages under $f|_V$ of a generic point from $f(V)$.
\item[(ii)] Otherwise, $f_*([V]) = 0$.
\end{itemize}
\end{lemma}

\begin{proof}
This is essentially \cite[Lem. 19.1.2]{FULTON} combined the remark in 
Section 1.4 of \cite{FULTON} that $\deg(V/f(V))$ is equal to the topological degree ${\rm deg}_{\rm top}(f|_{V})$
of $f|_V:~V~\rightarrow~f(V)$.
\end{proof}

\begin{lemma}\label{LEM:COMPONENTS_CONTROLLED_DIMENSION}
Let $X$ and $Z$ be $k$-dimensional algebraic manifolds and let $W$ be a $k$-dimensional subvariety of $X \times Z$.
We have
\begin{eqnarray*}
\pr_{1*}([W] \cupprod \pr_2^* \alpha) = 0
\end{eqnarray*}
if either
\begin{enumerate}
\item[(i)]  $\dim(\pr_1(W)) \leq k-p$ and $\alpha \in H^{i}(Z)$ for some $i < 2p$, or
\item[(ii)]$\dim(\pr_2(W)) \leq p$ and $\alpha \in H^{i}(Z)$ for some $i > 2p$.
\end{enumerate}
\end{lemma}

\begin{proof}
Suppose $\dim(\pr_1(W)) \leq k-p$.
The fundamental homology class $\{W\}$ is in the image of $\iota_\#$, where $\iota: \pr_1(W) \times Z \hookrightarrow X \times
 Z$ is the inclusion.  Therefore,
\begin{eqnarray*}
K^{-1}(\{W\}) =  \sum_{a=1}^{2k-2p} \sum_{l=1}^{l_a} g_{a,l} \otimes e_{2k-a,l},
\end{eqnarray*}
with each $g_{a,l} \in H_a(X)$ and each $e_{2k-a,l} \in H_{2k-a}(Z)$.  Applying Lemma \ref{LEM:KUNNETH}, we have
\begin{eqnarray*}
\kappa^{-1}([W]) = \sum_{a=1}^{2k-2p} \sum_{l=1}^{l_a} \PD_X^{-1}(g_{a,l}) \otimes \PD_Y^{-1}(e_{2k-a,l}) 
= \sum_{a=1}^{2k-2p} \sum_{l=1}^{l_a} \gamma_{2k-a,l} \otimes \eta_{a,l}.
\end{eqnarray*}
where each $\gamma_{2k-a,l} \in H^{2k-a}(X)$ and $\eta_{a,l} \in H^{a}(Z)$.  Thus for any $\alpha \in H^{i}(Z)$ we have,
\begin{eqnarray*}
[W] \cupprod \pr_2^* \alpha  = \sum_{a=1}^{2k-2p} \sum_{l=1}^{l_a} \pr_1^* (\gamma_{2k-a,l}) \cupprod \pr_2^* (\eta_{a,l} \cupprod \alpha) \\
= \kappa^{-1}\left(\sum_{a=1}^{2k-2p} \sum_{l=1}^{l_a} \gamma_{2k-a,l} \otimes (\eta_{a,l} \cupprod \alpha)\right).
\end{eqnarray*}
Since each term in the second factor has degree $2k-2p+i < 2k$, Lemma \ref{LEM:KUNNETH_PROJECTION} gives that
\begin{eqnarray*}
\pr_{2*}([W] \cupprod \pr_2^* \alpha) = 0.
\end{eqnarray*}
The proof of (ii) is essentially the same.
\end{proof}

\section{Alternative definitions for $f^*$ and remarks about $\Gamma_f$}
\label{SEC:SINGULAR_GRAPHS}

In this section, we'll show that two common alternative definitions for $f^*
\alpha$ are consistent with (\ref{EQN:PULLBACK}).    In Example \ref{EG:HENON}, we'll see that the graph $\Gamma_f$
may be singular at points whose first coordinate is in $I_f$.  For this reason, these alternative definitions for $f^* \alpha$
are more commonly used in actual computations.

\begin{lemma}\label{LEM:RESOLUTION1}
Suppose $\tilde{X}$ is a $k$-dimensional algebraic manifold and that $\pr: \tilde{X} \rightarrow X$
and  $\tilde{f} : \tilde{X}  \rightarrow Y$ are holomorphic maps making the Diagram (\ref{GENERAL_RESOLUTION})
commute.  Then, $f^* \alpha = \pr_* \left(\tilde{f}^* \alpha \right)$.
\end{lemma}

\noindent
Usually, $\pr: \tilde{X} \rightarrow X$ will be a blow-up, but Lemma \ref{LEM:RESOLUTION1} holds in greater generality.

\begin{proof}
This follows from the fact that $(\pr \times \id)_*\left[\Gamma_{\tilde{f}}\right] = [\Gamma_
f]$ and the Push-Pull formula.
\end{proof}

\begin{lemma}\label{LEM:RESOLUTION2}
Suppose that $\tilde{\Gamma_f}$ is a resolution of the singularities in $\Gamma_f$ and $\tilde{\pi_1}$ and $\tilde{\pi_2}$ are
the lifts of $\pi_1 |_{\Gamma_f}$ and $\pi_2 |_{\Gamma_f}$ to $\tilde{\Gamma_f}$.  Then, $f^* \alpha = \tilde{\pi}_{1*}(\tilde{\pi_2}^* \alpha)$.
\end{lemma}

\begin{proof}
This is a restatement of Lemma \ref{LEM:RESOLUTION1}.
\end{proof}

The following lemma will be helpful later.

\begin{lemma}\label{LEM:GRAPH_MINUS_VARIETY}
Let $X$ and $Y$ be algebraic manifolds of dimension $k$ and
let $f: X \dashrightarrow Y$ be a rational map.  If $V \subset X$ be a proper subvariety of $X$,
then
\begin{eqnarray*}
\Gamma_f = \overline{\{(x,y) \in X \times Y \, : \, x \not \in V \cup I_f \, \, \mbox{and} \, \, y = f(x) \}}.
\end{eqnarray*}
\end{lemma}

\begin{proof}
Since $\Gamma_f$ is defined by (\ref{EQN:GRAPH}), it suffices to show that $\Gamma_f \setminus \pi_1^{-1}(V)$
is dense in $\Gamma_f$.  This follows since
$\pi_1: \Gamma_f \rightarrow X$ is dominant, giving that $\pi_1^{-1}(V)$ is a proper subvariety of $\Gamma_f$.
\end{proof}

\begin{example}\label{EG:HENON}
Both Eric Bedford and the one of the anonymous referees have pointed out to us several
complications arising when working directly with the graph $\Gamma_f$.  The
following is an expanded version of Example 1 from \cite[\S 5]{BEDFORD_KOREA}.

The quadratic H\'enon map
$h_{a,c} : \mathbb{C}^2 \rightarrow \mathbb{C}^2$, given by
\begin{eqnarray}\label{EQN:HENON}
h_{a,c}(x_1,x_2) = (x_1^2+c - ax_2,x_1),
\end{eqnarray}
extends as a birational map of $\P^2$, which is expressed in homogeneous coordinates as
\begin{eqnarray}\label{EQN:HENON_HOMOG}
h_{a,c}([X_1:X_2:X_3]) = [X_1^2 + c X_3^2 -a X_2 X_3: X_1 X_3: X_3^2].
\end{eqnarray}
The extension has indeterminacy $I_h = \{[0:1:0]\}$.  One can check that $I_h$ blows-up under $h$ to $L_\infty := \{X_3 = 0\}$
and that $h(L_\infty \setminus I_h) = [1:0:0]$.  In particular, points of $\Gamma_h$ satisfy that $X_3 = 0 \Leftrightarrow Y_3 = 0$.

In $\mathbb{C}^2 \times \mathbb{C}^2$, the graph of $h$ is given by 
\begin{eqnarray}\label{EQN:HENON_AFF}
y_1 = x_1^2+c - ax_2 \qquad \mbox{and} \qquad y_2 = x_1.
\end{eqnarray}
It is natural to expect that the graph $\Gamma_h$ of $h: \P^2 \dashrightarrow \P^2$
is obtained by substituting \\ $x_1=\frac{X_1}{X_3}, x_2 = \frac{X_2}{X_3}, y_1 = \frac{Y_1}{Y_3}$, and $y_2 = \frac{Y_2}{Y_3}$ and then clearing denominators.  
One obtains 
\begin{eqnarray}\label{EQN:BADHENON}
Y_1 X_3^2 &=& X_1^2 Y_3 + c X_3^2 Y_3 - a X_2 X_3 Y_3 \qquad \mbox{and}\\
Y_2 X_3 &=& X_1 Y_3, \nonumber
\end{eqnarray}
which describe some subset of $\P^2 \times \P^2$.  
However, if one sets $X_3 = Y_3 = 0$, both equations become $0=0$, so that
$[X_1:X_2:0] \times [Y_1:Y_2:0]$ satisfies (\ref{EQN:BADHENON}) for any $X_1, X_2, Y_1$, and $Y_2$.  Thus, (\ref{EQN:BADHENON}) does not capture the
fact that $h(L_\infty \setminus I_h) = [1:0:0]$.

One can try adding further equations that are consistent with
(\ref{EQN:HENON_AFF}) on $\mathbb{C}^2 \times \mathbb{C}^2$.    When $X_3$ and $Y_3$ are not
$0$, it follows from the second equation in (\ref{EQN:BADHENON}) that
$\frac{X_1}{X_3} = \frac{Y_2}{Y_3}$, which implies $\frac{1}{X_3} = \frac{Y_2}{Y_3 X_1}$.
Substituting $x_1 = \frac{X_1 Y_2}{Y_3 X_1} = \frac{Y_2}{Y_3}, x_2 = \frac{X_2 Y_2}{Y_3 X_1}$ and $y_1 = \frac{Y_1}{Y_3}$ into first equation from (\ref{EQN:HENON_AFF}), clearing denominators, and dividing by a common factor of $X_1$ adds a third equation to the system:
\begin{eqnarray}\label{EQN:GOODHENON}
Y_1 X_3^2 &=& X_1^2 Y_3 + c X_3^2 Y_3 - a X_2 X_3 Y_3, \nonumber \\
Y_2 X_3 &=& X_1 Y_3, \qquad \mbox{and} \\
Y_1 Y_3 X_1 &=& X_1 Y_2^2 + c Y_3^2 X_1 - a X_2 Y_2 Y_3. \nonumber
\end{eqnarray}
When one substitutes $X_3 = Y_3 = 0$ into (\ref{EQN:GOODHENON}), the third equation becomes $0 = X_1 Y_2^2$, which expresses
that $h(L_\infty \setminus I_h) = [1:0:0]$.  However, $[0:1:0] \times [Y_1:0:Y_3]$ satisfies (\ref{EQN:GOODHENON}) for any $Y_1$ and $Y_3$.  Thus (\ref{EQN:GOODHENON}) does not imply that $X_3 = 0 \Leftrightarrow Y_3 = 0$, which is required for points of~$\Gamma_h$. 

If one computes $h^{-1} : \mathbb{C}^2 \rightarrow \mathbb{C}^2$, one of the equations is $x_2 = \frac{1}{a}(y_2^2 +c - y_1)$.  Converting this equation to homogeneous
coordinates and adding it to our system, we obtain
\begin{eqnarray}\label{EQN:GREATHENON}
Y_1 X_3^2 &=& X_1^2 Y_3 + c X_3^2 Y_3 - a X_2 X_3 Y_3, \nonumber \\
Y_2 X_3 &=& X_1 Y_3, \\
Y_1 Y_3 X_1 &=& X_1 Y_2^2 + c Y_3^2 X_1 - a X_2 Y_2 Y_3, \qquad  \mbox{and} \nonumber \\
a X_2 Y_3^2 &=& X_3 Y_2^2 + c X_3 Y_3^2 - Y_1 Y_3 X_3. \nonumber
\end{eqnarray}
These four equations imply that $X_3 = 0 \Leftrightarrow Y_3 = 0$.  When $X_3 = Y_3 = 0$, the third
equation becomes $0 = X_1 Y_2^2$, which describes $\Gamma_h \cap (L_\infty \times L_\infty)$ and when $X_3 \neq 0$ and $Y_3 \neq 0$
the first two equations describe $\Gamma_h \cap (\mathbb{C}^2 \times \mathbb{C}^2)$.  Therefore, (\ref{EQN:GREATHENON}) describes
$\Gamma_h \subset \P^2 \times \P^2$.

One might wonder whether $\Gamma_h$ can be described with fewer equations.
This is related to the notion on {\em complete intersection}; see, for example, 
\cite[Exercise I.2.17]{HART}.  If we let $J$ be
the ideal in $\mathbb{C}[X_1,X_2,X_3,Y_1,Y_2,Y_3]$ defined by
(\ref{EQN:GREATHENON}), one can compute $I(\Gamma_h) = \sqrt{J}$ using the
computer algebra package Macaulay2 \cite{MAC}.  One finds that $I(\Gamma_h)$ is
generated by two equations

\begin{eqnarray}\label{EQN:COMPLETE}
X_3 Y_2 - X_1 Y_3 &=& 0, \qquad \mbox{and} \\
X_3 Y_1 - X_1 Y_2 + a X_2 Y_3 - c X_3 Y_3 &=& 0. \nonumber
\end{eqnarray}
Thus, $\Gamma_h$ is a complete intersection.  In particular, these two equations describe $\Gamma_h$ in all of $\P^2 \times \P^2$.

If we express (\ref{EQN:COMPLETE}) in the local coordinates $z_1 = X_1/X_2, z_2 = X_3/X_2, w_1 = Y_2/Y_1,$ and $w_2 = Y_3/Y_1$ centered at $[0:1:0] \times [1:0:0]$, we find
\begin{eqnarray*}
z_2 w_1 - z_1 w_2 &=& 0 \qquad  \mbox{and} \\
z_2 - z_1 w_1 + a w_2 - c z_2 w_2 &=& 0.
\end{eqnarray*}
Since the lowest order terms of the first equation are quadratic, when one restricts $\Gamma_h$ to any plane through $(z_1,z_2,w_1,w_2) = (0,0,0,0)$, the result will have local multiplicity $\geq 2$.  This implies that $\Gamma_h$ has local multiplicity $\geq 2$
at $[0:1:0] \times [1:0:0]$ and hence that $\Gamma_h$ is singular there.
\end{example}

\section{Proof of the composition formula}\label{SEC:FUNCTOR}

The proof of (\ref{EQN:COMPOSED_PULLBACK0}) below is cribbed from Voisin's textbook \cite[Prop. 9.17]{VOISIN2}.

\begin{proof}{Proof of Proposition \ref{PROP:FUNCTOR}:}
For any $\alpha \in H^*(Z)$ we have
\begin{eqnarray*}
\pr_{1*}(\rho_{2*}(\rho_1^*[\Gamma_f] \cupprod \rho_3^* [\Gamma_g]) \cupprod \pr_2^* \alpha) &\stackrel{\rm PP}{=}& \pr_{1*}(\rho_{2*}(\rho_1^*[\Gamma_f] \cupprod \rho_3^* [\Gamma_g] \cupprod \rho_2^* \pr_2^* \alpha))  \\
= \pr_{1*}(\rho_{2*}(\rho_1^*[\Gamma_f] \cupprod \rho_3^* [\Gamma_g] \cupprod (\pi_4 \circ \rho_3)^* \alpha))
&=& \pi_{1*}(\rho_{1*}(\rho_1^*[\Gamma_f] \cupprod \rho_3^* [\Gamma_g] \cupprod (\pi_4 \circ \rho_3)^* \alpha))\\
\stackrel{\rm PP}{=} \pi_{1*}([\Gamma_f] \cupprod \rho_{1*}(\rho_3^* [\Gamma_g] \cupprod (\pi_4 \circ \rho_3)^* \alpha)) &=&
\pi_{1*}([\Gamma_f] \cupprod \rho_{1*} (\rho_3^*([\Gamma_g] \cupprod \pi_4^* \alpha))) \\ \stackrel{\Diamond}{=} \pi_{1*}([\Gamma_f] \cupprod \pi_{2}^* (\pi_{3*}([\Gamma_g] \cupprod \pi_4^* \alpha)))  &=& f^* g^* \alpha.
\end{eqnarray*}
Here, all unlabeled equalities follow from commutativity of Diagram (\ref{MASSIVE_DIAGRAM}) and the equality labeled PP follows from the Push-Pull formula.
To check $\Diamond$, one must show for any $\beta \in H^*(Y \times Z)$ that
\begin{eqnarray}
(\rho_{1*} \circ \rho_3^*) \beta = (\pi_2^* \circ \pi_{3*}) \beta.
\end{eqnarray}
This follows easily by expanding $\beta$ using the K\"unneth formula and applying Lemma  \ref{LEM:KUNNETH_PROJECTION}.



\vspace{0.1in}
We'll now check that if $\rho_{2*}(\rho_1^*[\Gamma_f] \cupprod \rho_3^* [\Gamma_g]) \neq [\Gamma_{g
\circ f}]$, then $f^* g^* \neq (g \circ f)^*$.
Let
\begin{eqnarray*}
\rho_{2*}(\rho_1^*[\Gamma_f] \cupprod \rho_3^* [\Gamma_g]) = [\Gamma_{g \circ f}] + \mathcal{E}.
\end{eqnarray*}
By linearity of (\ref{EQN:COMPOSED_PULLBACK0}), it suffices to find some $\alpha \in H^*(Z)$ with $\pr_{1*}(\mathcal{E} \cupprod \pr_2^* \alpha) \neq 0$. 
For each $i=0, \ldots,2k$, let $\gamma_{i,1},\ldots,\gamma_{i,j_i}$ be a basis of $H^i(X)$ and let $\eta_{i,1},\ldots,\gamma_{i,l_i}$ be a basis of $H^i(Z)$.  Using the K\"unneth Isomorphism, we have
\begin{eqnarray*}
\kappa^{-1}(\mathcal{E}) = \sum_{i=0}^{2k} \sum_{j=1}^{j_i} \sum_{l=1}^{l_{2k-i}} a_{i,j,l} \gamma_{i,j} \otimes \eta_{2k-i,l}.
\end{eqnarray*}
Since $\mathcal{E} \neq 0$, there is some 
$a_{i_0,j_0,l_0} \neq 0$.  Since we are using field coefficients the cup product
is a duality pairing; see \cite[Ch. VI, Thm. 9.4]{BREDON}.
We can therefore find some $\alpha \in H^{i_0}(Z)$ so that $\eta_{2k-i_0,l_0} \cupprod
\alpha = [z_\bullet]$ and $\eta_{2k-i_0,l} \cupprod \alpha = 0$ for every $l \neq l_0$.  (Here, $[z_\bullet]$ is the fundamental cohomology class of a point $z_\bullet \in Z$ and a generator of $H^{2k}(Z)$.)
This implies that
\begin{eqnarray*}
\kappa^{-1}(\mathcal{E} \cupprod \pr_2^* \alpha) &=& \sum_{i=0}^{2k} \sum_{j=1}^{j_i} \sum_{l=1}^{l_{2k-i}} a_{i,j,l} \gamma_{i,j} \otimes (\eta_{2k-i,l} \cupprod \pr_2^* \alpha)  \\ &=& a_{i_0,j_0,l_0} (\gamma_{i_0,j_0} \otimes [z_\bullet]) + \sum_{i=i_0+1}^{2k} \sum_{j=1}^{j_i} \sum_{l=1}^{l_{2k-i}} a_{i,j,l} \gamma_{i,j} \otimes (\eta_{2k-i,l} \cupprod \pr_2^* \alpha).
\end{eqnarray*}
Lemma \ref{LEM:KUNNETH_PROJECTION} implies
\begin{eqnarray*}
\pr_{1*}(\mathcal{E} \cupprod \pr_2^* \alpha) = a_{i_0,j_0,l_0} \gamma_{i_0,j_0} \neq 0.
\end{eqnarray*}
we conclude that $f^* \circ g^* (\alpha) \neq (g \circ f)^* (\alpha)$.
\end{proof}

\section{Criteria for functoriality} \label{SEC:CRITERIA}

We'll now start our study of the intersection $\rho_1^{-1}(\Gamma_f) \cap \rho_3^{-1}(\Gamma_g)$.  Let
\begin{eqnarray}\label{EQN:U}
U:=\{(x,y,z) \in X \times Y \times Z \, : \, (x,y) \in \Gamma_f,\, x \not \in I_f, \, (y,z) \in \Gamma_g, \, \mbox{and} \, y \not \in I(g) \}.
\end{eqnarray}
\begin{lemma}\label{LEM:PRINCIPAL_COMP}
We have
\begin{itemize}
\item[(i)] $\rho_1^{-1}(\Gamma_f)$ and $\rho_3^{-1}(\Gamma_g)$ are smooth and intersect transversally  at points of $U$ and
\item[(ii)] $V = \overline{U}$ is an irreducible component of $\rho_1^{-1}(\Gamma_f) \cap \rho_3^{-1}(\Gamma_g) \subset X \times Y \times Z$ that is mapped to $\Gamma_{g \circ f}$ by $\rho_2$ with topological degree $1$.
\end{itemize}
Consequently, if $U$ is dense in $\rho_1^{-1}(\Gamma_f) \cap \rho_3^{-1}(\Gamma_g)$ then $(g \circ f)^* = f^* \circ g^*$ 
on all cohomology.
\end{lemma}

We will call $V$ the {\em principal component} of  $\rho_1^{-1}(\Gamma_f) \cap \rho_3^{-1}(\Gamma_g)$.

\begin{proof}
Since $x \not \in I_f$ and $y \not \in I_g$, $\rho_1^*(\Gamma_f)$ and $\rho_3^*(\Gamma_g)$ are smooth at any $(x,y,z) \in U$.
For any $(x,y,z) \in U$ we have
\begin{eqnarray*}
T_{(x,y,z)} \rho_1^{-1}(\Gamma_f) &=& \{(u_1,Df_x u_1, w_1) \, : \, u_1 \in T_x X \, \mbox{and} \, w_1  \in T_z Z \} \, \mbox{and} \\
T_{(x,y,z)} \rho_3^{-1}(\Gamma_g) &=& \{(u_2,v_2, Dg_y v_2) \, : \, v_2 \in T_y Y \, \mbox{and} \, w_2  \in T_z Z \}.
\end{eqnarray*}
Therefore, $T_{(x,y,z)} \rho_1^{-1}(\Gamma_f) + T_{(x,y,z)} \rho_3^{-1}(\Gamma_g) = T_{(x,y,z)} X \times Y \times Z$, so that
$\rho_1^{-1}(\Gamma_f)$ and $\rho_3^{-1}(\Gamma_g)$ are transverse at $(x,y,z)$.

\vspace{0.1in}
Notice that $\rho_2(U)$ is the graph of $(g\circ f)_{|X \setminus (I_f \cup {f^{-1}(I(g))}}$, which is
dense in $\Gamma_{g \circ f}$, by Lemma~\ref{LEM:GRAPH_MINUS_VARIETY}.  Since $\rho_2$ is continuous and closed,
\begin{eqnarray}
\rho_2\left(V\right) = \rho_2\left(\overline{U} \right)  = \overline{\rho_2(U)} = \Gamma_{g\circ f}.
\end{eqnarray}
Finally, notice that $\rho_2: U \rightarrow \rho_2(U)$ is one-to-one since for points of $U$, $x$ completely determines $y$ and $z$.  In particular, since $\Gamma_{g\circ f}$ is irreducible, so is $V$.

If $U$ is dense in $\rho_1^{-1}(\Gamma_f) \cap \rho_3^{-1}(\Gamma_g)$, then by Lemmas \ref{LEM:CUP_PRODUCT} and \ref{LEM:PUSH_FORWARD_VARIETIES} we have
\begin{eqnarray}
\rho_{2*}([\rho_1^{-1} (\Gamma_f)] \cupprod [\rho_3^{-1} (\Gamma_g)]) = \rho_{2*}([V]) = [\Gamma_{g\circ f}].
\end{eqnarray}
It follows from Proposition \ref{PROP:FUNCTOR} that $(g \circ f)^* = f^* \circ g^*$ on all cohomology.
\end{proof}

Let us also prove one more helpful lemma:

\begin{lemma}\label{LEM:STABLE_EXTREME_CASES}
Let $X,Y,Z$ be algebraic manifolds of dimension $k$ and let $f: X \dashrightarrow Y$ and $g: Y \dashrightarrow Z$ be
rational maps.
If  $\alpha \in H^i(Z)$ for $i \in \{0,1,2k-1,2k\}$,
then $(g \circ f)^* \alpha = (f^* \circ g^*) \alpha$.
\end{lemma}

\begin{proof}
Since $f$ is dominant, $f^{-1}(I_g) \cup I_f$ is a proper subvariety of $X$.  Thus, 
any irreducible component of $\rho_1^{-1}(\Gamma_f) \cap
\rho_3^{-1}(\Gamma_g)$ that projects under $\pi_1 \circ \rho_1$ onto all of $X$
is equal to the principal component $V$.
Similarly, since $g$ is dominant $g(I_g \cup f(I_f))$ is a proper subvariety of $Z$, implying
that any irreducible component of $\rho_1^{-1}(\Gamma_f) \cap
\rho_3^{-1}(\Gamma_g)$ that projects under $\pi_4 \circ \rho_3$ onto all of $Z$
is equal to the principal component $V$.

Thus, if $W \neq \Gamma_{g \circ f}$ is a $k$-dimensional subvariety of
$\rho_2(\rho_1^{-1}(\Gamma_f) \cap \rho_3^{-1}(\Gamma_g))$ whose fundamental
class appears in the expression for $\rho_{2*}(\rho_1^*[\Gamma_f] \cup \rho_3^*
[\Gamma_g])$, one finds that $\dim(\pr_1(W)) \leq k-1$ and $\dim(\pr_2(W)) \leq
k-1$.   It then follows from Proposition \ref{LEM:COMPONENTS_CONTROLLED_DIMENSION}
that if $i \in \{0,1,2k-1,2k\}$ and
$\alpha \in H^i(Z)$ then $\pr_{1*}([W] \cupprod \pr_2^* \alpha) = 0$.
Equation (\ref{EQN:COMPOSED_PULLBACK0}) then implies that $(g \circ f)^* \alpha = (f^* \circ g^*) \alpha$.
\end{proof}

We are now ready to prove Propositions \ref{PROP:FINITE_FIBERS} --
\ref{PROP:DS}.  For the reader's convenience we'll repeat the
statements before each of the proofs.

\begin{proposition12} 
Let $f: X \dashrightarrow Y$ and $g: Y \dashrightarrow Z$ be rational maps.
Suppose that there exits an algebraic manifold $\tilde{X}$ and holomorphic maps $\pr$ and $\tilde{f}$ making the following
diagram commute (wherever $f \circ \pr$ is defined)
\begin{equation} 
\xymatrix{\tilde{X} \ar[d]^\pr  \ar[dr]^{\tilde{f}}  & \\
X \ar @{-->}[r]^f & Y } \tag{\ref{GENERAL_RESOLUTION}}
\end{equation}
with the property that $\tilde{f}^{-1}(x)$ is a finite set for every $y \in Y$.  Then, $(g\circ f)^* = f^* \circ g^*$ on all cohomology groups.
\end{proposition12}

\begin{proof}
By Lemma \ref{LEM:PRINCIPAL_COMP}, it suffices to show that $U$, given
by (\ref{EQN:U}), is dense in  $\rho_1^{-1}(\Gamma_f)\cap\rho_3^{-1}(\Gamma_g)$.

Consider any $(x_\bullet,y_\bullet,z_\bullet) \in \rho_1^{-1} (\Gamma_f) \cap
\rho_3^{-1} (\Gamma_g)$.  We'll show that $(x_\bullet,y_\bullet,z_\bullet)$ is
the limit of a sequence $\{(x_n,y_n,z_n)\} \subset U$.  
Since $f(I_f) := \pi_2(\pi_1^{-1}(I_f) \cap \Gamma_f)$ is a proper subvariety of $Y$, 
Lemma \ref{LEM:GRAPH_MINUS_VARIETY} gives that 
$\Gamma_g$ is the closure
of the graph of $g|_{Y \setminus (f(I_f) \cup I_g)}$.  Therefore, we can choose a sequence 
$$\{(y_n,z_n)\} \in \{Y \times Z \, : \, (y,z) \in \Gamma_g \, \mbox{and} \, y
\not \in (f(I_f) \cup I_g)\}$$ with $(y_n,z_n) \rightarrow (y_\bullet,z_\bullet)$.  

Since $(x_\bullet,y_\bullet) \in \Gamma_f$, there exists $\tilde{x}_\bullet \in \tilde{X}$ with $\pr(\tilde{x}_\bullet) = x_\bullet$ and $\tilde{f}(\tilde{x}_\bullet) = y_\bullet$.  Since $\tilde{f}$ is a finite map, it is open.  Therefore
we can choose a sequence of preimages $\tilde{x}_n$ of $y_n$ under $\tilde{f}$ with $\tilde{x}_n \rightarrow \tilde{x}_\bullet$.
If we let $x_n = \pr(\tilde{x_n})$, by continuity of $\pr$ we have $x_n \rightarrow x_\bullet$.
Since $y_n \not \in f(I_f)$ we have that each $x_n \not \in I_f$.  Therefore, we have found a sequence $(x_n,y_n,z_n) \in U$ with $(x_n,y_n,z_n) \rightarrow (x_\bullet,y_\bullet,z_\bullet)$.
\end{proof}

\begin{proposition14}
{\rm (Diller-Favre \cite[Prop. 1.13]{DILLER_FAVRE})} Let $X,Y,$ and $Z$ be algebraic manifolds of dimension $2$.  Then
$(g\circ f)^* = f^* \circ g^*$ if and only if there is no curve $C \subset X$
with $f(C \setminus I_f) \subset I_g$.   
\end{proposition14}

\begin{remark}
In the case that $(g\circ f)^* \neq f^* \circ g^*$, it follows from Lemma \ref{LEM:STABLE_EXTREME_CASES} that
the discrepancy happens on $H^2(Z)$.
\end{remark}

\begin{proof}
Suppose that there is no curve $C$ with $f(C \setminus I_f) \subset I_g$.   By
Lemma \ref{LEM:PRINCIPAL_COMP}, it suffices to show that the set $U$, given by
(\ref{EQN:U}), is dense in $\rho_1^{-1}(\Gamma_f) \cap
\rho_3^{-1}(\Gamma_g)$.

Let $(x_\bullet,y_\bullet,z_\bullet) \in \rho_1^{-1} (\Gamma_f) \cap
\rho_3^{-1} (\Gamma_g)$. 
If $y_\bullet \not \in I_g$, then we can choose a sequence $\{(x_n,y_n)\} \subset \Gamma_f$ converging to $(x_\bullet,y_\bullet)$ with
each $x_n \not \in I_f$.  Since $y_\bullet \not \in I_g$ and $I_g$ is closed, $y_n \not \in I_g$ for large enough $n$.  Letting
$z_n = g(y_n)$, we obtain a sequence $\{(x_n,y_n,z_n)\} \subset U$ which converges to $(x_\bullet,y_\bullet,z_\bullet)$.

Now, suppose $y_\bullet \in I_g$.
As in the proof of Proposition \ref{PROP:FINITE_FIBERS} we will use that $\Gamma_g$ is the closure
of the graph of $g|_{Y \setminus (f(I_f) \cup I_g)}$.
Therefore, we can choose a sequence
$\{(y_n,z_n)\} \subset \Gamma_g$ with $(y_n,z_n) \rightarrow
(y_\bullet,z_\bullet)$ and each $y_n \not \in f(I_f) \cup I_g$.  We must show
that there is a sequence $x_n \in X \setminus I_f$ with $f(x_n) = y_n$ and $x_n
\rightarrow x_\bullet$.

Since $X$ is a
surface, we can make a resolution of indeterminacy of the form
(\ref{GENERAL_RESOLUTION}) where $\pr$ consists of a sequence of point blow-ups
over $I_f$.  Since $(x_\bullet,y_\bullet) \in \Gamma_f$, there exists
$(\tilde{x}_\bullet,y_\bullet) \in \Gamma_{\tilde{f}}$ with
$\pr(\tilde{x_\bullet}) = x_\bullet$.  Let $D$ be the component of
$\tilde{f}^{-1}(y_\bullet)$ containing $\tilde{x}_\bullet$.  Since $y_\bullet
\neq f(C \setminus I_g)$ for any curve $C \subset X$, $\pr(D) = \{x_\bullet\}$.

Since $\tilde{f}(D) = y_\bullet$ and $D$ is a component of
$\tilde{f}^{-1}(y_\bullet)$, we can choose a sequence $\tilde{x_n} \in
\tilde{X}$ with $\tilde{f}(\tilde{x}_n) = \tilde{y}_n$ such that $\tilde{x_n}
\rightarrow D$.  Since $\pr(D) = x_\bullet$, the desired sequence $x_n \in X \setminus I_f$ is $x_n = \pr(\tilde{x}_n)$.

\vspace{0.1in}
Now, suppose that there are curves  $C_1,\ldots,C_m \subset X$ with $\{y_i\} := f(C_i \setminus I_f)
\subset I_g$ for each $i$.  
For each $i$, $g(y_i) = D_i$ is a curve in $Z$.  Then,
\begin{eqnarray*}
\rho_1^{-1} (\Gamma_f) \cap \rho_3^{-1} (\Gamma_g) = V \cup \bigcup_{i=1}^m C_i \times \{y_i\} \times D_i
\end{eqnarray*}
with each term in the union being an independent irreducible component and $V= \overline{U}$ being the principal component.

Since each component has complex dimension $2$, by Lemma \ref{LEM:CUP_PRODUCT} 
\begin{eqnarray*}
\rho_1^*[\Gamma_f] \cupprod \rho_3^* [\Gamma_g] = [V] + \sum_{i=1}^m a_i [C_i \times \{y_i\} \times D_i],
\end{eqnarray*}
where each $a_i > 0$ is a suitable intersection number.

By Lemmas  \ref{LEM:PRINCIPAL_COMP} and \ref{LEM:PUSH_FORWARD_VARIETIES},
\begin{eqnarray*}
\rho_{2*}\left([V] + \sum_{i=1}^m a_i [C_i \times \{y_i\} \times D_i]\right) = [\Gamma_{g\circ f}] + \sum_{i=1}^m a_i [C_i \times D_i].
\end{eqnarray*}
Since $\sum_{i=1}^m a_i [C_i \times D_i] \neq 0$, it follows from Proposition \ref{PROP:FUNCTOR} that $(g \circ f)^* \neq f^* \circ g^*$.
\end{proof}



\begin{remark} \label{REM:NON_PROPER_INTERSECTION}
Up to this point, we have only needed the simple cases (i) and (ii) of
Lemma~\ref{LEM:CUP_PRODUCT} in which the subvarieties intersect with the
correct dimension.  The proofs of the criteria of Bedford-Kim and Dinh-Sibony
below rely upon case (iii) of Lemma \ref{LEM:CUP_PRODUCT}, since one can easily
have components of $\rho_1^{-1}(\Gamma_f) \cap \rho_3^{-1}(\Gamma_g)$ of
dimension $> k$.  For example, if $f: \mathbb{CP}^3 \dashrightarrow
\mathbb{CP}^3$ and $g: \mathbb{CP}^3 \dashrightarrow \mathbb{CP}^3$ are both
Cremona involutions
\begin{eqnarray}\label{EQN:CREMONA}
[x_1:x_2:x_3:x_4] \mapsto [x_2 x_3 x_4 : x_1 x_3 x_4: x_1 x_2 x_4: x_1 x_2 x_3]
\end{eqnarray}
then $\rho_1^{-1}(\Gamma_1) \cap \rho_3^{-1}(\Gamma_g)$ contains a four-dimensional component
\begin{eqnarray*}
\{x_0 = 0\} \times \{[1:0:0:0]\} \times \{z_0 = 0\}.
\end{eqnarray*}
\end{remark}

\begin{proposition15}
\rm (Bedford-Kim \cite[Thm. 1.1]{BEDFORD_KIM2})
Let $X,Y,$ and $Z$ be  algebraic manifolds of dimension $k$.
Let $f: X \dashrightarrow Y$ and $g: Y \dashrightarrow Z$ be rational maps.   If there is no hypersurface $H$
with $f(H \setminus I_f) \subset I_g$, then $f^* \circ g^* = (g\circ f)^*$ on $H^2(Z)$.
\end{proposition15}

\begin{proof}
By Lemmas \ref{LEM:CUP_PRODUCT} and \ref{LEM:PUSH_FORWARD_VARIETIES}, 
\begin{eqnarray}\label{EQN:COMPOSED_CLASS}
\rho_{2*}(\rho_1^*[\Gamma_f] \cupprod \rho_3^* [\Gamma_g]) = [\Gamma_{g \circ f}] + \sum a_i [W_i],
\end{eqnarray}
where each $W_i$ is a $k$-dimensional subvariety of $X \times Z$ and each $a_i \in \mathbb{Z}$.  
By Lemma \ref{LEM:COMPONENTS_CONTROLLED_DIMENSION}, it suffices to show for every $i$ that
$\dim(\pr_1(W_i)) < k-1$.  

Suppose for some $i = i_0$ that 
$\dim(\pr_1(W_{i_0})) \geq k-1$.  Then, by
commutativity of (\ref{MASSIVE_DIAGRAM}), $V_{i_0} := \pr_2^{-1}(W_{i_0})$ satisfies
that $\dim(\pi_1 \circ \rho_1(V_{i_0})) \geq k-1$.  
The hypothesis that there is no hypersurface $H$ with $f(H \setminus I_f) \subset I_g$ implies that $\dim(f^{-1}(I_g)) \leq k-2$.  Therefore $\dim(I_f \cup f^{-1}(I_g)) \leq k-2$.
Thus, there is a dense set of points $(x,y,z) \in V_{i_0}$ with $x \not \in I_f \cup f^{-1}(I_g)$.  All such points 
are in the principal component $V$; therefore $V = V_{i_0}$. This contradicts that $\rho_2(V_{i_0}) = W_{i_0} \neq \Gamma_{g \circ f} = \rho_2(V)$.
We conclude that, $\dim(\pr_1(W_i)) \leq k-2$ for every $i$.
\end{proof}

\begin{remark}
Recently, Bayraktar \cite[Thm. 5.3]{BAYRAKTAR} has proved that the Bedford-Kim criterion is necessary, i.e. if 
there is a hypersurface $H \subset X$ with $f(H \setminus I_f) \subset I_g$, then $(g\circ f)^* \neq f^* \circ g^*$
on $H^2(Z)$.

This does not seem to follow from the results developed in this note,
since,
when $k = dim(X) \geq 3$, $\rho_1^{-1}(\Gamma_f) \cap \rho^{-1}(\Gamma_g)$ may have
components of dimension $> k$.  (See Remark~\ref{REM:NON_PROPER_INTERSECTION}.)  
For this reason, the cup product $\rho_1^* [\Gamma_f] \cupprod \rho_3^* [\Gamma_g]$ may be represented by some $k$-dimensional subvarieties having negative coefficients.    In particular, one must prove that the cohomology classes from all of
the extra components of $\rho_1^{-1}(\Gamma_f) \cap \rho^{-1}(\Gamma_g)$ don't completely cancel.
\end{remark}

\begin{remark} 
There is an older criterion of Bedford and Kim \cite[Prop. 1.2]{BEDFORD_KIM}, which one can
check is strictly weaker than the one stated in Proposition \ref{PROP:BK}.
\end{remark}

Recall that  $\tilde{\Sigma}'_f \subset \Gamma_f$ is the set of points such that 
\begin{enumerate}
\item[(i)] $\pi_2$ restricted to $\Gamma_f$ is not locally finite at $x$, and
\item[(ii)] $\pi_2((x,y)) \in I_g$ for every $(x,y) \in \tilde{\Sigma}'_f$. 
\end{enumerate}
 Let $\Sigma'_f:=\pi_1\left(\tilde{\Sigma}'_f\right)$.

\begin{proposition16}
{\rm (Variant of Dinh-Sibony \cite[Prop. 5.3.5]{DINH_SIBONY}) } Let $X,Y,$ and $Z$ be algebraic manifolds of dimension $k$.
Let $f: X \dashrightarrow Y$ and $g: Y \dashrightarrow Z$ be rational maps.  If $\dim \Sigma'_f < k-p$, then $(g\circ f)^* = f^* \circ g^*$
on $H^i(Z)$ for $1 \leq i \leq 2p$.
\end{proposition16}

\begin{proof}
As in the proof of Proposition \ref{PROP:BK}, $\rho_{2*}(\rho_1^*[\Gamma_f] \cupprod \rho_3^* [\Gamma_g])$
can be expressed by (\ref{EQN:COMPOSED_CLASS}).  By Lemma \ref{LEM:COMPONENTS_CONTROLLED_DIMENSION} it suffices to show 
for every $i$ that $\dim(\pr_1(W_i)) < k-p$.

Suppose for some $i = i_0$ that
$\dim(\pr_1(W_{i_0})) \geq k-p$.  Then, by
commutativity of (\ref{MASSIVE_DIAGRAM}), $V_{i_0} := \pr_2^{-1}(W_{i_0})$ satisfies
that $\dim(\pi_1 \circ \rho_1(V_{i_0})) \geq k-p$.  We'll show that the set $U$, given by (\ref{EQN:U}), is dense in $V_{i_0}$,
implying that $V_{i_0} = V$.   This will contradict that $\rho_2(V_{i_0}) = W_{i_0} \neq \Gamma_{g \circ f} = \rho_2(V)$.

Since $\dim(\Sigma'_f) < k-p$, we have that $V_{i_0} \setminus (\pi_1 \circ
\rho_1)^{-1}(\Sigma'_f)$ is dense in $V_{i_0}$.  Therefore, it suffices to show
that $U$ is dense in $V_{i_0} \setminus (\pi_1 \circ \rho_1)^{-1}(\Sigma'_f)$.
Let $$(x_\bullet,y_\bullet,z_\bullet) \in V_{i_0} \setminus (\pi_1 \circ
\rho_1)^{-1}(\Sigma'_f).$$  First suppose that $y_\bullet \not \in I_g$.  Then, we can choose a sequence $(x_n,y_n) \in \Gamma_f$
with $x_n \not \in I_f$ converging to $(x_\bullet,y_\bullet)$.  Since $y_\bullet \not \in I_g$ and $I_g$ is closed, $y_n \not \in I_g$ for large enough $n$.  Thus, if we let $z_n = g(y_n)$, we obtain a sequence $(x_n,y_n,z_n) \in U$ that converges to $(x_\bullet,y_\bullet,z_\bullet)$.

Now suppose that $y_\bullet \in I_f$.  
By Lemma \ref{LEM:GRAPH_MINUS_VARIETY}, $\Gamma_g$ is the closure of the
graph of $g|_{Y \setminus (I_g \cup f(I_f))}$.  Thus, we can find a sequence
$(y_n,z_n)$ in the graph of $g|_{Y \setminus (I_g \cup f(I_f))}$ with
$(y_n,z_n) \rightarrow (y_\bullet,z_\bullet)$.  Meanwhile, since $x_\bullet~\not \in~\Sigma'_f$ and $y_\bullet \in I_g$,
$\pi_2|\Gamma_f$ is a finite map in a neighborhood of $(x_\bullet,y_\bullet)$.
It follows from the Weierstrass Preparation Theorem that $\pi_2|_{\Gamma_f}$ is an open map
in that neighborhood.  Therefore, there is a sequence $x_n \in X$ with
$(x_n,y_n) \in \Gamma_f$ and $(x_n,y_n)~\rightarrow~(x_\bullet,y_\bullet)$.
Since $y_n \not \in f(I_f)$, $x_n \not \in I_f$.  Thus, we have found a
sequence $\{(x_n,y_n,z_n)\} \subset U$ with $(x_n,y_n,z_n) \rightarrow
(x_\bullet,y_\bullet,z_\bullet)$.  We conclude that $U$ is dense in $V_{i_0}$.
\end{proof}

\begin{remark} The reader who is interested in proving Proposition \ref{PROP:DS} using currents should note that Truong \cite{TRUONG} presents
an approach to pulling back $(p,p)$-currents for $p > 1$ that is somewhat different from \cite{DINH_SIBONY}.  In particular,  one can also use Theorem 7 from \cite{TRUONG} to prove Proposition \ref{PROP:DS}.
\end{remark}

\section{A cautionary example}\label{SEC:CONCLUSION}

In this section, we present a rational map $f:X \dashrightarrow X$ of a
three-dimensional algebraic manifold $X$ that is not $2$-stable, but has the property
that $\left(f|_{X \setminus I_f}\right)^{-1}(x)$ is a finite set for every $x \in X$.  

Let $f_0: \P^3 \dashrightarrow \P^3$ be the composition $f_0 = \alpha_0 \circ s_0$, where 
$\alpha_0$ is the birational map
{\small
\begin{align}
& \alpha_0([x_1:x_2:x_3:x_4]) \\ & = [x_1(x_2-x_4)(x_3-x_4):x_4(x_2-x_4)(x_3-x_4):x_4(x_2-x_1)(x_3-x_4):x_4(x_2-x_4)(x_3-x_1)] \nonumber
\end{align}
}
and $s_0$ is the squaring map
\begin{align}
s_0([x_1:x_2:x_3:x_4]) = [x_1^2:x_2^2:x_3^2:x_4^2].
\end{align}
Let $\varrho: X \rightarrow \P^3$ be the blow-up of $\P^3$ at the five points $[1:0:0:0],[0:1:0:0],[0:0:1:0],[0:0:0:1],$ and $[1:1:1:1]$ and let 
$\alpha: X \dashrightarrow X$, $s: X \dashrightarrow X$, and $f$ be the lifts of $\alpha_0$, $s_0$, and $f_0$.

\begin{proposition} $f: X \dashrightarrow X$ satisfies for every $x \in X$ that $\left(f|_{X \setminus I_f}\right)^{-1}(x)$ is a finite set, but $f$ is not $2$-stable.
\end{proposition}

\begin{proof}
We will start by showing that $\left(f|_{X \setminus I_f}\right)^{-1}(x)$ is a finite set for every $x \in X$.
One can check that
$\alpha_0: \P^3 \dashrightarrow \P^3$ satisfies
\begin{eqnarray*}
I_{\alpha_0} =&& \{x_1=x_4=0\} \cup \{x_2=x_4=0\} \cup \{x_3= x_4=0\} \\ &&\cup \{x_1=x_2=x_4\} \cup \{x_1=x_3=x_4\}
\cup \{x_2= x_3 = x_4\}.
\end{eqnarray*}
The four points $[1:0:0:0],[0:1:0:0],[0:0:1:0],[1:1:1:1]$ where these six lines meet are each blown-up up by $\alpha_0$
to the following hyperplanes
\begin{eqnarray*}
\{x_2=0\}, \{x_2=x_3\}, \{x_2=x_4\}, \mbox{  and  } \{x_1 = x_2\},
\end{eqnarray*}
respectively.
Once these four points are removed, the six lines in $I_{\alpha_0}$ map by flip indeterminacy (see, for example, \cite{kollar2008birational}) to the lines
\begin{eqnarray*}
&& \{x_2=x_3=x_4\} \cup \{x_2=x_4=0\} \cup \{x_2= x_3=0\} \\ && \cup \{x_1=x_2=x_4\} \cup \{x_1=x_2=x_3\}
\cup \{x_1= x_2 = x_4\},
\end{eqnarray*}
respectively.   A flip indeterminacy from a line $L_1$ to a line $L_2$ blows up
any point $p \in L_1$ to all of $L_2$, based on the direction that one
approaches $p$ within a transversal plane.  Meanwhile, there is a {\em collapsing
behavior}: if one approaches any two points of $L_1$ with the same transversal
direction, one is sent by $f$ to the same point of $L_2$.  

The critical set of $\alpha_0$ consists of the hypersurfaces
\begin{eqnarray*}
\{x_4=0\} \cup \{x_1=x_4\} \cup \{x_2 = x_4\} \cup \{x_3=x_4\},
\end{eqnarray*}
which are all collapsed by $\alpha_0$ to the points
$[1:0:0:0],[1:1:1:1],[0:0:1:0],$ and $[0:0:0:1]$, respectively, with the first
three of these points in $I_{\alpha_0}$. 

When creating $X$, we have blown-up each of the images of the varieties that are collapsed by $\alpha_0$ as well
as $[0:0:0:1] = \alpha_0^{-1}([0:1:0:0])$.
Using the universal property of blow-ups \cite{EISENBUD_HARRIS} (or direct calculations), one can check the indeterminacy set of
$\alpha: X \dashrightarrow X$ is the proper transform of $I_{\alpha_0}$ and that $\alpha$ has no critical points (outside
of $I_\alpha$).  In particular, $\alpha$ collapses no curves or hypersurfaces lying outside of $I_\alpha$.

One can check that $s: X \dashrightarrow X$ is holomorphic in a neighborhood of
each of the exceptional divisors, inducing the squaring map on each of the
exceptional divisors $E_{[1:0:0:0]}$, $E_{[0:1:0:0]}$, $E_{[0:0:1:0]}$, $E_{[0:0:0:1]}$ and is
the identity on the exceptional divisor  $E_{[1:1:1:1]}$.  As a result, the
only indeterminacy points of $s$ are the lifts under $\varrho$ above the $7$
points in \\ $s_0^{-1}([1:1:1:1]) \setminus [1:1:1:1]$, each of which is blown-up
by $s$ to $E_{[1:1:1:1]}$.  Since $s_0$ did not collapse and
hypersurfaces or curves and $s$ does not collapse anything in these exceptional
divisors,  we conclude that $s$ also doesn't collapse and hypersurfaces or
curves.

The indeterminacy of $f = \alpha \circ s$ is contained in $s^{-1}(I_\alpha) \cup I_s$.
One can check that $I_s \subset I_f$ since $\alpha$ doesn't collapse $E_{[1:1:1:1]}$.
Meanwhile,
$s^{-1}(I_\alpha)$ is
the proper transform of the 15 lines
\begin{eqnarray*}
s_0^{-1}(I_{\alpha_0}) =&& \{x_1=x_4=0\} \cup \{x_2=x_4=0\} \cup \{x_3= x_4=0\} \\ &&\cup \{x_1=\pm x_2= \pm x_4\} \cup \{x_1=\pm x_3=\pm x_4\}
\cup \{x_2= \pm x_3 = \pm x_4\}.
\end{eqnarray*}
Taking any point from one of these $15$ lines that is not on one of the exceptional divisors,
one can use the homogeneous expression for $f_0$ to see that each such point is in $I_f$.  
Therefore, $I_f = s^{-1}(I_\alpha) \cup I_s$.  Since neither $\alpha$ nor
nor $s$ collapse any variety outside of their indeterminate sets, we conclude that $f$ 
doesn't collapse any variety outside of $I_f$.

\vspace{0.1in}

The only non-trivial cohomology groups of $X$ are 
\begin{eqnarray*}
H^{0}(X) \cong \mathbb{C}, H^2(X) = H^{(1,1)}(X)~\cong~\mathbb{C}^6, H^4(X) = H^{(2,2)}(X) \cong
\mathbb{C}^6, \, \mbox{and} \,  H^{6}(X) \cong \mathbb{C}.
\end{eqnarray*}  By Lemma
\ref{LEM:STABLE_EXTREME_CASES}, any rational map acts stably on $H^{0}(X)$ and
$H^{6}(X)$.  Since $I_f$ is of codimension $\geq 2$ and $f$ collapses nothing
outside of $I_f$, the Bedford-Kim criterion (Prop.  \ref{PROP:BK}) implies that
$f$ acts stably on $H^{(1,1)}(X)$.  Therefore, in order to prove that $f$ is
not $2$-stable, it suffices to show that $(f^2)^* \neq (f^*)^2$ on $H^*(X)$.  

For notational
convenience, we'll write the composition as $g \circ f$, where $X=Y=Z$ and $f:
X \dashrightarrow Y$ and $g: Y \dashrightarrow Z$ are the same map.  We will
first show that every component of $\rho_1^{-1}(\Gamma_f) \cap
\rho_3^{-1}(\Gamma_g)$ has the correct dimension ($= 3$), so that Lemma
\ref{LEM:CUP_PRODUCT} implies that each component appears with positive
multiplicity in the cup product $\rho_1^*([\Gamma_f]) \cupprod
\rho_3^*([\Gamma_g])$.  We will then show that there is at least one component
$V'$ of the intersection other than the principal component $V$ (see Lemma
\ref{LEM:PRINCIPAL_COMP} for the definition of the principal component) with
the property that $\rho_{2*}([V']) \neq 0$.  Non-functoriality $(f^2)^* \neq (f^*)^2$ will then follow from
Proposition \ref{PROP:FUNCTOR}.

Consider a component $V' \neq V$ of $\rho_1^{-1}(\Gamma_f) \cap
\rho_3^{-1}(\Gamma_g)$.  Since $V' \neq V$ and $f$ a finite map outside of
$I_f$, $(\pi_1 \circ \rho_1)(V') \subset I_f$.  Meanwhile, $(\pi_2 \circ
\rho_1)(V') \subset I_g$.  Let $\phi: V' \rightarrow \phi(V') \subset X \times
Y$ be the restriction of the projection onto the first two coordinates.  If
$\dim(V') \geq 4$, then, since $\dim(\phi(V')) \leq 2$, the fibers over generic
points satisfy $\dim(\phi^{-1}(x,y)) \geq 2$.  However, there are finitely many
points $y \in Y$ with $\dim(g(y)) \geq 2$, implying that $\dim(\pi_2 \circ
\rho_1)(V') = 0$.  Therefore, the dimension of the projections of $V'$ onto
$X$, $Y$, and $Z$ would be $1, 0$, and $2$, respectively, implying that
$\dim(V') = 3$.  

We will now find a component $V' \neq V \subset \rho_1^{-1}(\Gamma_f) \cap
\rho_3^{-1}(\Gamma_g)$ so that $\rho_{2*}([V']) \neq 0$.  
Let $L_1 \subset X$ be the proper transform of $\{x_1 = x_3 = x_4\}$, let $L_2 \subset Y$
be the proper transform of $\{y_1 = y_2 = y_3\}$, let $p = \varrho^{-1}([1:1:1:-1])
\in L_2$, and let $H \subset Z$ be the proper transform of $\{z_1 = z_2\}$.
Since $s$ maps a neighborhood of $L_1$ biholomorphically onto
a neighborhood $L_1$ and $\alpha$ blows-up each point of $L_1$ to all of $L_2$,  we have
that $L_1 \times \{p\} \subset \Gamma_f$.   Meanwhile, since $s$ blows up the
indeterminate point $p$ to $E_{[1:1:1:1]}$, which is mapped by $\alpha$ to the
plane $z_1 = z_2$, we have that $\{p\} \times H \subset \Gamma_g$.  We conclude
that
\begin{eqnarray*}
V':= L_1 \times \{p\} \times H \, \subset \, \rho_1^{-1}(\Gamma_f) \cap \rho_3^{-1}(\Gamma_g).
\end{eqnarray*}
Since $V'$ is an irreducible $3$-dimensional variety with $(\pi_1 \circ
\rho_1)(V) = L_1 \subsetneq X$, it is not the principal component $V$.
{\em It is the collapsing behavior of the flip indeterminacy along $L_1$ into the point of indeterminacy $p$ that produces this extra component of
$\rho_1^{-1}(\Gamma_f) \cap \rho_3^{-1}(\Gamma_g)$.  This will lead to non-functoriality of the composition.}

Since $\rho_2$ maps $V'$ biholomorphically to $L_1 \times H  \subset X
\times Z$,  $\rho_{2*}([V']) \neq 0$.
We conclude that $(g \circ f)^* \neq f^* \circ g^*$.
\end{proof}

\begin{remark}
In a joint work with S. Koch \cite{KR}, we check that $f$ can be lifted to a
further blow-up of $X$ on which Proposition \ref{PROP:FINITE_FIBERS} can be
applied.  We then compute that the first and second dynamical degrees of this
lift satisfy that $\lambda_1 \approx 2.3462$ is the largest root of  
\begin{eqnarray*}
p_1(z) = z^4 - z^3 - 4z - 8,
\end{eqnarray*}
$\lambda_2 \approx 4.6658$ is the largest root of 
\begin{eqnarray*}
p_2(z) = z^9-3z^8-16z^6-192z^5+384z^4+128z^3+6144z-8192.
\end{eqnarray*}

Since dynamical degrees are invariant under birational conjugacy,
these are the same as the dynamical degrees of $f: X \dashrightarrow X$.
However, one can check that $\dim(H^4(X)) = 6$.
Since $p_2$ is an irreducible polynomial of degree $9 > 6$, this gives an
alternate proof that $f:X \dashrightarrow X$ is not $2$-stable.
\end{remark}

\begin{question}
Does there exist a rational map $f: \mathbb{P}^3 \dashrightarrow \mathbb{P}^3$ such that $\left(f|_{\mathbb{P}^3 \setminus I_f}\right)^{-1}(x)$ is a finite set for every $x \in \mathbb{P}^3$, that is not $2$-stable?
\end{question}

\section*{Acknowledgements}
Much of this work is inspired by the techniques from \cite{AMERIK} and by the results of the other authors mentioned above.  We have also benefited from discussions with
Eric Bedford, Olguta Buse, Jeffrey Diller,  Kyounghee Kim, Sarah Koch, Patrick Morton, and Tuyen Truong.
Finally, we thank the referees for their several helpful suggestions.
This work is supported in part by the NSF grant DMS-1102597
and startup funds from the Department of Mathematics at IUPUI.

\end{document}